\documentclass{amsart}
\usepackage{enumerate}

\newtheorem{theorem}{Theorem}[section]
\newtheorem{lemma}[theorem]{Lemma}
\newtheorem{corollary}[theorem]{Corollary}
\theoremstyle{definition}

\theoremstyle{remark}
\newtheorem{remark}[theorem]{Remark}

\numberwithin{equation}{section}

\begin{document}

\title[Hessian Bounds for Heat Equations on K\"ahler Manifolds]
{Upper Bounds for Hessian Matrices of Positive Solutions to Heat Equations on K\"ahler Manifolds}

\author{Zhao Wenshuo}
\address{School of Mathematics and Statistics, Wuhan University, Wuhan, Hubei Province, China}

\email{wenshuo@whu.edu.cn}
\thanks{The first author was supported in part by NSF Grant \#000000.}

\author{Zhu Anqiang}
\address{School of Mathematics and Statistics, Wuhan University, Wuhan, Hubei Province, China}
\email{aqzhu.math@whu.edu.cn}
\thanks{Support information for the second author.}


\date{}


\begin{abstract}
We prove global and local upper bounds for the Hessian matrices of positive solutions to the heat equation on K\"ahler manifolds whose bisectional curvature is bounded from below. We also improve a result of Han and Zhang by weakening the curvature assumptions in their Hessian estimates on Riemannian manifolds. More precisely, we extend their global and local upper bounds, originally obtained under two-sided curvature bounds, to Riemannian manifolds with sectional curvature bounded from below.
\end{abstract}

\maketitle

\section{Introduction}
\label{}

In this paper, we study upper bounds for the Hessian matrices of positive solutions to the heat equation on Riemannian and K\"ahler manifolds. We establish both global and local forms of these estimates.

In the seminal paper \cite{Li-Yau}, Peter Li and S.-T. Yau established the fundamental gradient estimate, now known as the Li-Yau estimate, for positive solutions to the heat equation on Riemannian manifolds with Ricci curvature bounded from below. They also showed that the classical Harnack inequality follows from this gradient estimate by integration along minimizing geodesics. Later, in \cite{Hamilton}, Richard Hamilton extended the Li-Yau estimate to a matrix Hessian estimate under the stronger assumptions that the sectional curvature is bounded from below and that the covariant derivative of the Ricci curvature satisfies $|\nabla \operatorname{Ric}|<C$. In \cite{Cao-Ni}, Huai-Dong Cao and Lei Ni showed that, on K\"ahler manifolds, the assumption on the covariant derivative of the Ricci curvature can be removed, and that the sectional curvature assumption in Hamilton's result can be replaced by a lower bound on the holomorphic bisectional curvature. Many related Harnack estimates for positive solutions to the heat equation under geometric flows have also been established; see, for example, \cite{CaoXD}, \cite{Cao-Hamilton}, and \cite{Li-Zhang}.

The estimates discussed above are global in nature and mainly concern lower bounds for the Laplacian or Hessian of the logarithm of a positive solution to the heat equation. Upper bounds for the Hessian have also been studied. In \cite{Lijiayu}, Jiayu Li obtained estimates for the second derivatives of the heat kernel. In \cite{Han-Zhang}, Qing Han and Qi S. Zhang derived upper bounds for the Hessian matrices of logarithmic heat solutions on manifolds. They proved both global and local versions of these bounds, which are sharp in the corresponding settings, for fixed metrics and under the Ricci flow. In the fixed metric case, their results require a bound on the covariant derivative of the Ricci curvature and an $L^{\infty}$ bound on the curvature operator.

Motivated by Hamilton's matrix estimate, it is natural to expect global upper bounds for the Hessian matrices of positive heat solutions under the weaker assumptions that the sectional curvature is bounded from below and $|\nabla \operatorname{Ric}|<C$. We show that the method of Han and Zhang \cite{Han-Zhang} can be refined to yield both global and local upper bounds under these weaker curvature assumptions. Our first main result is the following.
\begin{theorem}\label{Riemannian}
Let $(M,g)$ be a Riemannian manifold with sectional curvature bounded from below, $\operatorname{sec}\geq -K_{1}$, where $K_{1}\geq 0$. Suppose that the covariant derivative of the Ricci curvature is bounded, $|\nabla \operatorname{Ric}|\leq K_{2}$. 
 \begin{enumerate}[(a)]
 \item Let $u(x,t)$ be a bounded positive solution of the heat equation, $0<u\leq A$,
\begin{align}\label{Heat equation}
\partial_{t}u-\Delta u=0, \quad (x,t)\in M\times [0,T).
\end{align}
Then 
\begin{align*}
t(u_{ij})\leq u(C+Bt)(1+\log \frac{A}{u}), \quad (x,t)\in M\times [0,T),
\end{align*}
where $C$ is a universal constant and $B$ is a nonnegative constant depending on $K_{1}$ and $K_{2}$.

 \item Suppose that $u(x,t)$ is a solution of 
\begin{align*}
\partial_{t}u-\Delta u=0, \quad (x,t)\in Q_{R,T}(x_{0},t_{0}).
\end{align*}
Then 
\begin{align*}
(u_{ij})\leq Cu\left(\frac{1}{T}+\frac{1}{R^{2}}+B\right)\left(1+\log \frac{A}{u}\right)^{2}, \quad (x,t)\in Q_{R,T}(x_{0},t_{0}),
\end{align*}
where $C$ is a universal constant and $B$ is a nonnegative constant depending only on the lower bound for the Riemannian curvature tensor and the bound for the covariant derivative of the Ricci curvature in $B\left(x_0, R\right)$.
  \end{enumerate}
\end{theorem}

\begin{remark}
Note that if $K_{1}=K_{2}=0$, then $B=0$. Thus, on a sphere with constant sectional curvature, we have 
\begin{align*}
tu_{ij}\leq c(n)u(1+\log \frac{u}{A})\leq c(n)A.
\end{align*}
As $t\rightarrow \infty$, it follows that $u_{ij}\leq \frac{c}{t}$.
\end{remark}

On Riemannian manifolds, the estimates above still require a bound on $|\nabla \operatorname{Ric}|$. The condition $\nabla \operatorname{Ric}=0$ would be too restrictive, since it imposes a strong parallel Ricci condition on the manifold. Motivated by \cite{Cao-Ni}, we next consider local and global Hessian estimates for positive solutions to the heat equation on K\"ahler manifolds. We show that, in the K\"ahler setting, the bound on the covariant derivative of the Ricci curvature can be removed, and that the $L^{\infty}$ bound on the Riemannian curvature can be replaced by a lower bound on the holomorphic bisectional curvature. Our second main result is the following.
\begin{theorem}\label{Kahler}
    Let $M$ be a complex manifold equipped with a K\"ahler metric $\omega$. Assume that the bisectional curvature of $(M,\omega)$ is bounded from below by $-K$, that is, $\operatorname{bisec} \geq -K$.
    \begin{enumerate}[(a)]
        \item Suppose that $u$ is a solution of
    \[
    \partial_t u-\Delta u=0 \quad \text{in } M \times(0, T] .
    \]
    
    Assume $0<u \leq A$. Then
    \[
    t\left(u_{i \bar{j}}\right) \leq u(C+K t)\left(1+\log \frac{A}{u}\right) \quad \text{in } M \times(0, T],
    \]
    where $C$ is a universal constant.
    \item Suppose that $u$ is a solution of
    \[
    \partial_t u-\Delta u=0 \quad \text{in } Q_{R, T}\left(x_0, t_0\right) .
    \]
    
    Assume $0<u \leq A$. Then
    \[
    \left(u_{i \bar{j}}\right) \leq C u\left(\frac{1}{T}+\frac{1}{R^2}+K\right)\left(1+\log \frac{A}{u}\right)^2 \quad \text{in } Q_{\frac{R}{2}, \frac{T}{2}}\left(x_0, t_0\right) .
    \]
    where $C$ is a universal constant.
    \end{enumerate}
\end{theorem}

When the K\"ahler manifold has nonnegative bisectional curvature, the estimates reduce to the following simpler form.

\begin{corollary}
    Let $(M,\omega)$ be a K\"ahler manifold with nonnegative bisectional curvature, and let $u$ be a positive bounded solution of the heat equation with $0<u\leq A$. Then
\begin{align*}
\left(u_{i \bar{j}}\right)\leq CA\frac{1}{t},
\end{align*}
where $C$ is a constant.
\end{corollary}

The paper is organized as follows. In Section~2, we introduce the notation used throughout the paper. In Section~3, we prove the K\"ahler case, Theorem~\ref{Kahler}. In Section~4, we prove the Riemannian case, Theorem~\ref{Riemannian}.

\section{Notation} 
Throughout the paper, $(M, \omega)$ denotes a K\"ahler manifold. In local coordinates $(z^1, \ldots, z^n)$, the Christoffel symbols are denoted by $\Gamma_{ij}^{k}=g^{k\bar{p}}\partial_i{g_{j\bar{p}}}$. On a K\"ahler manifold, the curvature tensor is defined by $R_{i \bar{j}k}^l =-\partial_{\bar{j}} \Gamma_{i k}^l$, and the Ricci curvature is $R_{i \bar{j}}=R_{i \bar{j} k \bar{l}} g^{k \bar{l}}$. The Hessian of a function $u$ is written as $u_{i \bar{j}}$. If $V$ is a $(1,1)$-Hermitian symmetric tensor field on $M$ and $\xi$ is a vector field of type $(1,0)$, then in local coordinates we use $\bar{\xi}^T V \xi$ to denote $V(\xi, \bar{\xi})$.
On a Riemannian manifold $(M,g)$, we use the convention $R_{ij}=g^{kl}R_{iklj}$ in local coordinates, and the Hessian of a function $u$ is written as $u_{ij}$.
The distance function is denoted by $d(x,y)$. A geodesic ball is denoted by $B(x, r)$, where $x$ is a point in $M$ and $r$ is the radius. For $R>0$ and $T>0$, we define the parabolic cube by
\[
Q_{R, T}\left(x_0, t_0\right)=B\left(x_0, R\right) \times\left(t_0-T, t_0\right] .
\]
Positive constants are denoted by $C$, with or without indices, and may change from line to line.

\section{Heat Equations on K\"ahler Manifolds}

Let $M$ be a K\"ahler manifold with metric $g$ and $\Delta=\frac{1}{2}(\nabla_{i}\nabla_{\bar{i}}+\nabla_{\bar{i}}\nabla_{i})$. We consider a positive solution $u$ of the heat equation
\[
u_t=\Delta u \quad \text { in } M \times(0, \infty) .
\]
We assume
\[
0<u \leq A .
\]
Set
\[
f=\log \frac{u}{A} .
\]
Let $\left\{z^1, \ldots, z^n\right\}$ be a local holomorphic normal frame at a point, say $p \in M$. Then
\[
f_i=\frac{u_i}{u}, \quad f_{i \bar{j}}=\frac{u_{i \bar{j}}}{u}-\frac{u_i u_{\bar{j}}}{u^2},
\]
and hence
\[
f_t=\Delta f+\frac{1}{2}|\nabla f|^2 .
\]

The following lemma is a key point of the paper. In \cite{Han-Zhang}, the Ricci flow is used to cancel the gradient Ricci term. In the K\"ahler case, these terms vanish because of the symmetries of the curvature tensor on a K\"ahler manifold.

\begin{lemma}\label{lemma:2.1}(cf. Lemma 2.1 in \cite{Cao-Ni})
    Let $u$ satisfy the heat equation $(\partial_t-\Delta) u=0$. Then the Hessian of $u$ satisfies
        \[
        \left(\frac{\partial}{\partial t}-\Delta\right) u_{i \bar{j}}=R_{i \bar{j} l \bar{k}} u_{k \bar{l}} - \frac{1}{2}
          \left(
            R_{i \bar{m}} u_{m {\bar{j}}} 
            + R_{m\bar{j}} u_{i \bar{m}}
          \right) .
        \]
\end{lemma}

\begin{proof}
    In a holomorphic normal coordinate system,
\begin{align}\label{partial t Delta u}
        \left(\frac{\partial}{\partial t} u_{i \bar{j}}\right)=(\Delta u)_{i \bar{j}}
          &=\left(g^{k \bar{l}} u_{k \bar{l}}\right)_{i \bar{j}}
          =g^{k \bar{l}} u_{k \bar{l} i \bar{j}}+R_{i \bar{j} l \bar{k}} u_{k \bar{l}}.
\end{align}

Since 
\begin{align*}
           \nabla_{\bar{l}} \nabla_k u_{i \bar{j}}&=\nabla_{\bar{l}}\left(\partial_k u_{i \bar{j}}-\Gamma_{i k}^m u_{m \bar{j}}\right) \nonumber\\
          & =u_{i \bar{j} k \bar{l}}+R_{i \bar{l} k \bar{m}} u_{m{\bar{j}}},
 \end{align*}
 we get
\begin{align}\label{Delta u 1}
        \Delta u_{i \bar{j}}=\frac{1}{2}\left(\nabla_{\bar{m}} \nabla_m u_{i \bar{j}}+\nabla_m \nabla_{\bar{m}} u_{i \bar{j}}\right) =u_{i \bar{j} m \bar{m}}+\frac{1}{2}\left(R_{i \bar{l}} u_{l\bar{j}}+R_{l\bar{j}} u_{i \bar{l}}\right).
\end{align}
Combining (\ref{partial t Delta u}) and (\ref{Delta u 1}), we obtain 
\begin{align*}
 \left(\frac{\partial}{\partial t}-\Delta\right) u_{i \bar{j}}=R_{i \bar{j} l \bar{k}} u_{k \bar{l}} - \frac{1}{2}
          \left(
            R_{i \bar{m}} u_{m {\bar{j}}} 
            + R_{m\bar{j}} u_{i \bar{m}}
          \right) .
\end{align*}
\end{proof}

\begin{lemma}\label{lemma:2.2}
Set
\[v_{i \bar{j}}=\frac{u_{i \bar{j}}}{u(1-f)}, \quad v_{i j}=\frac{u_{i j}}{u(1-f)},
\]
Then
\[
  \begin{aligned}
    \left[-\partial_t  + \Delta-\frac{1}{u(1-f)} \nabla f \cdot \nabla\right] v_{i \bar{j}} =&\frac{1}{2} \frac{|\nabla f|^2}{1-f} v_{i \bar{j}} +\frac{1}{u(1-f)} \\
    & +\left(-R_{i \bar{j} l \bar{k}} u_{k \bar{l}}+\frac{1}{2} R_{i \bar{l}} u_{l\bar{j}}+\frac{1}{2} R_{l \bar{j}} u_{i \bar{l}}\right).
  \end{aligned}
\]
\end{lemma}

\begin{proof}
    Noting that 
      \[
        \begin{aligned}
            \partial_t(u(1-f))&=-u_{t}f,\\
            \partial_k(u(1-f))&=-u_{k}f,
        \end{aligned}
      \]
    we have
      \[
        \begin{aligned}
            \partial_t v_{i\bar{j}} &=\frac{u_{i\bar{j}t}}{u(1-f)}+\frac{u_{i\bar{j}}fu_{t}}{u^{2}(1-f)^{2}}\\
            \partial_k v_{i\bar{j}}&=\frac{u_{i\bar{j}k}}{u(1-f)}+\frac{u_{i\bar{j}}fu_{k}}{u^{2}(1-f)^{2}}.
        \end{aligned}
        \]
        Note that
        \[
          \begin{aligned}
            \Delta v_{i \bar{j}}= & \frac{1}{2}\left(\nabla_k \nabla_{\bar{k}} v_{i \bar{j}}+\nabla_{\bar{k}} \nabla_k v_{i \bar{j}}\right) \\
             =&\frac{\Delta u_{i \bar{j}}}{u(1-f)}
            +\frac{u_{i \bar{j}} f \Delta u}{u^2(1-f)^2}
            +\frac{f\left\langle\nabla u_{i \bar{j}}, \nabla u\right\rangle}{u^2(1-f)^2} \\
            & +\frac{1}{2} \frac{u_{i \bar{j}}\langle\nabla f, \nabla u\rangle}{u^2(1-f)^2}+\frac{u_{i \bar{j}} f^2\left|\nabla u\right|^2}{u^3(1-f)^3}
          \end{aligned}
      \]
      Using $u_k=u f_k$ and Lemma~\ref{lemma:2.1}, we obtain
      \[
        \begin{aligned} 
           (\Delta-\partial_t) v_{i \bar{j}}
          =&\frac{f\left\langle\nabla u_{i \bar{j}}, \nabla f\right\rangle}{u(1-f)^2}+\frac{1}{2} \frac{u_{i \bar{j}}\left|\nabla f\right|^2}{u(1-f)^2} \\
          &+\frac{u_{i \bar{j}} f^2|\nabla f|^2}{u(1-f)^3}+\frac{1}{u(1-f)}\left(-R_{i \bar{j} l \bar{k}} u_{k \bar{l}}+\frac{1}{2} R_{i \bar{l}} u_{l\bar{j}}+\frac{1}{2} R_{l \bar{j}} u_{i \bar{l}}\right)\\
          =&\frac{f}{1-f}\left\langle\nabla f, \frac{\nabla u_{i \bar{j}}}{u(1-f)}+\frac{u_{i \bar{j}} f \nabla f}{u^2(1-f)^2}\right\rangle \\
          & +\frac{1}{2} \frac{u_{i \bar{j}}}{u(1-f)} \frac{\left|\nabla f \right|^2}{(1-f)}+\frac{1}{u(1-f)}\left(-R_{i \bar{j} l \bar{k}} u_{k \bar{l}}+\frac{1}{2} R_{i \bar{l}} u_{l\bar{j}}+\frac{1}{2} R_{l \bar{j}} u_{i \bar{l}}\right),
        \end{aligned}
      \]
      Using the expression for $\partial_k v_{i \bar{j}}$, we obtain
      \[
        \begin{aligned}
          (\Delta-\partial_t)v_{i \bar{j}}
          =&\frac{f}{1-f}\langle\nabla f, \nabla v_{i \bar{j}}\rangle
          +\frac{1}{2}\frac{|\nabla f|^{2}}{1-f}v_{i \bar{j}}\\
          &+\frac{1}{u(1-f)}\left(-R_{i \bar{j} l \bar{k}} u_{k \bar{l}}+\frac{1}{2} R_{i \bar{l}} u_{l\bar{j}}+\frac{1}{2} R_{l \bar{j}} u_{i \bar{l}}\right).
        \end{aligned}
      \]
\end{proof}

\begin{lemma}
    Set \[w_{i \bar{j}}=\frac{u_i u_{\bar{j}}}{u^2(1-f)^2}, \quad w_{i j}=\frac{u_i u_j}{u^2(1-f)^2},\]
    then
    \[
        \begin{aligned}
         \left(-\partial_t+\Delta-\frac{f}{1-f} \nabla f \cdot \nabla\right) w_{i \bar{j}}=&\frac{|\nabla f|^2}{1-f} w_{i \bar{j}}+\left(v_{i \bar{k}}+f w_{i \bar{k}}\right)\left(v_{k \bar{j}}+f w_{k \bar{j}}\right) \\
        & +\left(v_{i k}+f w_{i k}\right) \overline{\left(v_{kj}+f w_{kj}\right)}+\frac{1}{2} R_{i \bar{k}} w_{k \bar{j}}+\frac{1}{2} R_{k \bar{j}} w_{\bar{k} i}.
        \end{aligned}
    \]
\end{lemma}

\begin{proof}
    The proof proceeds similarly to that of Lemma~\ref{lemma:2.2}. First,
      \[
          \begin{aligned}
          \partial_t w_{i\bar{j}} &=
            \frac{u_{it}u_{\bar{j}}+u_{i}u_{\bar{j}t}}{u^{2}(1-f)^{2}}
            +
            \frac{2u_{i}u_{t}u_{\bar{j}}f}{u^{3}(1-f)^{3}}\\
          \partial_k w_{i\bar{j}} &=\frac{u_{ik}u_{\bar{j}}+u_{i}u_{\bar{j}k}}{u^{2}(1-f)^{2}}+\frac{2u_{k}u_{i}u_{\bar{j}}f}{u^{3}(1-f)^3}.
          \end{aligned}
      \]
    Second, by Bochner's formula, we have
    \[
    \begin{aligned}
      \Delta w_{i \bar{j}}  
        =&\frac{(\Delta u)_i u_{\bar{j}}+\left\langle\nabla u_i, \nabla u_{\bar{j}}\right\rangle+u_i\left(\Delta u\right)_{\bar{j}}}{u^2(1-f)^2} \\
      & +\frac{1}{2} R_{i \bar{k}} \frac{u_k u_{\bar{j}}}{u^2(1-f)^2}
        +\frac{1}{2} R_{k \bar{j}} \frac{u_{\bar{k}} u_i}{u^2(1-f)^2}
        +\frac{2 u_i u_{\bar{j}} \Delta u f}{u^3(1-f)^3} \\
      & +\frac{u_i u_{\bar{j}}\left\langle\nabla u,\nabla f\right\rangle}{u^3 (1-f)^3}
        +\frac{2 f\big(\left\langle \nabla u_i, \nabla u\right\rangle u_{\bar{j}}+\left\langle\nabla u_{\bar{j}}, \nabla u\right\rangle u_i\big)}{u^3(1-f)^3} 
      +\frac{3 | \nabla u|^2 u_i u_{\bar{j}} f^2}{u^4(1-f)^4}
    \end{aligned}
    \]
    hence
    \[
      (\Delta-\partial_t) w_{i \bar{j}}=H+\frac{1}{2} R_{i \bar{k}} w_{k \bar{j}}+\frac{1}{2} R_{k \bar{j}} w_{i \bar{k}},
    \]
    where
    \[
      \begin{aligned}
        H =&\frac{
          \left(
            \left\langle\nabla u_i, \nabla u\right\rangle u_{\bar{j}}+\left\langle\nabla u_{\bar{j}}, \nabla u\right\rangle u_i
          \right)f
        }{u^3(1-f)^3}
          +\frac{2|\nabla u|^2 u_i u_{\bar{j}} f^2}{u^4(1-f)^4} +\frac{u_i u_{\bar{j}}\left\langle\nabla u, \nabla f\right\rangle}{u^3(1-f)^3}
          +\frac{\left\langle\nabla u_i, \nabla u_{\bar{j}}\right\rangle}{u^2(1-f)^2} \\
        & +\frac{
          \left(
            \left\langle\nabla u_i, \nabla u\right\rangle u_{\bar{j}}+\left\langle\nabla u_{\bar{j}}, \nabla u\right\rangle u_i
          \right)f
        }{u^3(1-f)^3}
          +\frac{u_i u_{\bar{j}} | \nabla u |^2 f^2}{u^4(1-f)^4} \\
        =& \frac{f}{u(1-f)}
          \left\langle
            \nabla u, \frac{\nabla u_i u_{\bar{j}}+u_i \nabla u_{\bar j}}{u^2(1-f)^2}
          +\frac{2 u_i u_{\bar{j}} f\nabla u }{u^3(1-f)^3}
          \right\rangle\\
        & + \frac{\langle\nabla u,\nabla f\rangle}{u(1-f)}\frac{u_i u_{\bar{j}}}{u^{2}(1-f)^{2}}
          +\left\langle
            \frac{\nabla u_i}{u(1-f)}
            +\frac{u_{i}f\nabla u}{u^{2}(1-f)^{2}}, \frac{\nabla u_{\bar{j}}}{u(1-f)}
            +\frac{u_{\bar{j}}f \nabla u}{u^{2}(1-f)^2}
            \right\rangle.
      \end{aligned}
    \]
  Using $u_k = uf_k$ and the expression for $\partial_k w_{i \bar{j}}$, we obtain 
  \[
  \begin{aligned}
    H= & \frac{f}{u(1-f)}\langle\nabla u, \nabla w_{i \bar{j}}\rangle
      +\frac{|\nabla f|^2}{1-f} w_{i \bar{j}} \\
    & +\left(v_{i \bar{k}}+f w_{i \bar{k}}\right)\left(v_{k \bar{j}} + f w_{k \bar{j}}\right)
      +\left(v_{i k}+f w_{i k}\right)\left(\overline{v_{k j}+f w_{k j}}\right). 
  \end{aligned}
  \]
\end{proof}

\begin{remark}
  We define the trace $w$ of $\left(w_{i \bar{j}}\right)$ by
  \[
  w=\operatorname{tr}\left(w_{i \bar{j}}\right)= \frac{|\nabla u|^2}{2u^2(1-f)^2}=\frac{1}{2}\frac{|\nabla f|^{2}}{(1-f)^{2}} .
  \]

  We also have
  \[
    \begin{aligned}
      (-\partial_t &+ \Delta-\frac{f}{1-f} \nabla f \cdot \nabla) v_{i \bar{j}} = (1-f) w v_{i \bar{j}}+\frac{1}{u(1-f)}(-R_{i \bar{j} l \bar{k}} u_{k \bar{l}}+\frac{1}{2} R_{i \bar{l}} u_{l\bar{j}}+\frac{1}{2} R_{l \bar{j}} u_{i \bar{l}}), \\
      (-\partial_t + \Delta & -\frac{f}{1-f} \nabla f \cdot \nabla ) w_{i \bar{j}} = 2(1-f) w w_{i \bar{j}}\\
&+\left(v_{i \bar{k}}+f w_{i \bar{k}}\right)\left(v_{k \bar{j}} + f w_{k \bar{j}}\right) +\left(v_{i k}+f w_{i k}\right)\left(\overline{v_{k j}+f w_{k j}}\right)+ \frac{1}{2} R_{i \bar{k}} w_{k \bar{j}}+\frac{1}{2} R_{k \bar{j}} w_{i \bar{k}}
    \end{aligned}
  \]
\end{remark}

We now turn to the proof of Theorem~\ref{Kahler}.

\begin{proof}[Proof of Theorem~\ref{Kahler}]
  \textbf{Part (a).} We begin with some preliminary calculations. Let $p$ be a point on the manifold, and let $\left\{z^1, \ldots, z^n\right\}$ be a local holomorphic normal coordinate system. In these coordinates, and with the notation of Lemmas~\ref{lemma:2.1} and~\ref{lemma:2.2}, the $(1,1)$-tensor fields $v_{i \bar{j}}$ and
$w_{i \bar{j}}$ may be viewed as $n \times n$ matrices. Set $V=\left(v_{i \bar{j}}\right), W=\left(w_{i \bar{j}}\right), V'=\left(v_{i j}\right), W'=\left(w_{i j}\right)$, $w=\operatorname{tr}(W)$, and
  \[
  L=-\partial_t+\Delta-\frac{f}{1-f} \nabla f \cdot \nabla.
  \]
  
  Then, by Lemma~\ref{lemma:2.1} and Lemma~\ref{lemma:2.2},
  \begin{align}
  L V & = (1-f) w V+P, \\
  L W & =2(1-f) w W
    + (\overline{V + fW})^T(V + fW) + (\overline{V' + fW'})^T(V' + fW')
    +Q,
  \end{align}
  where $P$ and $Q$ are matrices whose $(i, j)$ components are
  \begin{equation}
      P_{i \bar{j}} = -R_{i \bar{j} l \bar{k}} v_{k \bar{l}}+\frac{1}{2} R_{i \bar{l}} v_{l\bar{j}}+\frac{1}{2} R_{l \bar{j}} v_{i \bar{l}}
  \end{equation}
  and
\begin{equation}
  Q_{i \bar{j}}=\frac{1}{2} R_{i \bar{k}} w_{k \bar{j}}+\frac{1}{2} R_{k \bar{j}} w_{i \bar{k}}.
\end{equation}

For a constant $\alpha \in \mathbb{R}$ to be determined, we have
\[
\begin{aligned}
  L(\alpha V+W)=&\alpha(1-f) w V+2(1-f) w W+
  (\overline{V + fW})^T(V + fW) \\
&+ (\overline{V' + fW'})^T(V' + fW')
+\alpha P+Q.
\end{aligned}
\]

Let $\xi \in T^{(1,0)}_p M$ be a unit tangent vector at $p$. We extend $\xi$ to a smooth vector field in the local coordinate neighborhood by parallel translation along geodesics emanating from $p$, and still denote the resulting vector field by $\xi$. Since $V$ and $W$ are $(1,1)$-tensor fields, the function
\[
\lambda=\bar{\xi}^T(\alpha V+W) \xi \equiv(\alpha V+W)(\xi,\bar {\xi})
\]
is a well-defined smooth function in a neighborhood of $p$. Then
\[
L \lambda=H+\bar{\xi}^T(\alpha P+Q) \xi
\]
where
\begin{equation}\label{def H}
  H=\alpha(1-f) w \bar{\xi}^T V \xi+2(1-f) w \bar{\xi}^T W \xi+|(V+f W) \xi|^2 + |(V'+f W') \xi|^2 .
\end{equation}

By $\alpha \bar{\xi}^T V \xi=\lambda-\bar{\xi}^T W \xi$, we have
\begin{align}\label{H formula}
H & = (1-f) w\left(\lambda-\bar{\xi}^T W \xi\right)+2(1-f) w \bar{\xi}^T W \xi+|(V+f W) \xi|^2 + |(V'+f W') \xi|^2 \nonumber\\
& =(1-f) w \lambda+(1-f) w \bar{\xi}^T W \xi+|(V+f W) \xi|^2 + |(V'+f W') \xi|^2.
\end{align}

To simplify the last term further, we fix the point $p$ and assume that $\xi$ is obtained, by parallel translation along geodesics emanating from $p$, from an eigenvector of $\alpha V+W$ at $p$; that is, at $p$,
\[
(\alpha V+W) \xi=\lambda \xi.
\]
Then
\[
(V+f W) \xi=\frac{\lambda}{\alpha} \xi-\frac{1}{\alpha} W \xi+f W \xi=\frac{\lambda}{\alpha} \xi-\left(\frac{1}{\alpha}-f\right) W \xi,
\]
and hence
\[
|(V+f W) \xi|^2=\frac{\lambda^2}{\alpha^2}-\frac{2 \lambda}{\alpha}\left(\frac{1}{\alpha}-f\right) \bar{\xi}^T W \xi+\left(\frac{1}{\alpha}-f\right)^2|W \xi|^2.
\]
Hence
\[
  \begin{aligned}
  H=&\frac{\lambda^2}{\alpha^2}
    + \lambda\left(w-\frac{2}{\alpha^2} \bar{\xi}^T W \xi\right)
    -f \lambda\left(w-\frac{2}{\alpha} \bar{\xi}^T W \xi\right) \\
    &+(1-f) w \bar{\xi}^T W \xi
    +\left(\frac{1}{\alpha}-f\right)^2|W \xi|^2+|(V'+fW')\xi|^2
  \end{aligned}
\]
The last three terms are independent of $\lambda$ and are nonnegative. Hence
\begin{equation}\label{estimate H}
  H \geq \frac{\lambda^2}{\alpha^2}
    +\lambda\left(w-\frac{2}{\alpha^2} \bar{\xi}^T W \xi\right)
    -f \lambda\left(w-\frac{2}{\alpha} \bar{\xi}^T W \xi\right).
\end{equation}

For the second and third terms, we note that $W$ is a rank-one matrix; hence
\[
\bar{\xi}^T W \xi \leq w.
\]

Since $f<0$, if we choose $\alpha \geq 2$, then the second and third terms are nonnegative whenever $\lambda \geq 0$. Hence
\begin{equation}
  H \geq \frac{\lambda^2}{\alpha^2}.
\end{equation}

In summary, if $\lambda \geq 0$, then at the point $p$
\begin{equation}\label{L step 1}
  L \lambda \geq \frac{\lambda^2}{\alpha^2}+\bar{\xi}^T(\alpha P+Q) \xi.
\end{equation}

Let $\tau$ be a universal constant to be fixed later. With $\alpha\geq 2$, suppose the $(1,1)$-tensor field
\[
\alpha V+W-\frac{\tau}{t} g
\]
attains its largest nonnegative eigenvalue at a space time point $\left(p_1, t_1\right)$, with $t_1>0$. Let $\xi$ be a unit eigenvector at $p_1$ corresponding to the largest nonnegative eigenvalue of $\alpha V+W-\frac{\tau}{t} g$. 
We use parallel translation along geodesics emanating from $p_1$ to extend $\xi$ to a smooth vector field still denoted by $\xi$. Set, in local coordinates,
\begin{equation}
  \mu=\bar{\xi}^T\left(\alpha V+W-\frac{\tau}{t} g\right) \xi
\end{equation}
and
\begin{equation}\label{def lambda}
  \lambda=\bar{\xi}^T(\alpha V+W) \xi .
\end{equation}

Then both $\mu$ and $\lambda$ are smooth functions in a space time neighborhood of $\left(p_1, t_1\right)$. Moreover,
\[
L \mu=L\left(\lambda-\frac{\tau}{t}\right)=L \lambda-\frac{\tau}{t^2}=H-\frac{\tau}{t^2}+\bar{\xi}^T(\alpha P+Q) \xi
\]

Here $H$ is given by (\ref{H formula}). We now evaluate at $\left(p_1, t_1\right)$. Since $\lambda-\tau / t$ attains its nonnegative maximum at $\left(p_1, t_1\right)$, (\ref{L step 1}) gives
\[
0 \geq L\left(\lambda-\frac{\tau}{t}\right) \geq \frac{\lambda^2}{\alpha^2}-\frac{\tau}{t^2}+\bar{\xi}^T(\alpha P+Q) \xi \quad \text { at }\left(p_1, t_1\right)
\]
or
\begin{equation} \label{eq:2.11}
  0\geq\frac{\lambda^{2}}{\alpha^{2}}-\frac{\tau}{t^2}+\bar{\xi}^T(\alpha P+Q)\xi \quad \text { at }\left(p_1, t_1\right).
\end{equation}

We can rewrite $\alpha P+Q$ as
\begin{align}\label{alpha P+Q}
\alpha P+Q&=-R_{i\bar{j}l\bar{k}}(\alpha v_{k\bar{l}}+w_{k\bar{l}})+\frac{1}{2}R_{i\bar{l}}(\alpha v_{l\bar{j}}+w_{l\bar{j}})+\frac{1}{2}R_{l\bar{j}}(\alpha v_{i\bar{l}}+w_{i\bar{l}})+R_{i\bar{j}l\bar{k}}w_{k\bar{l}}
\end{align} 

Since, at $(p_{1},t_{1})$, $\xi$ is an eigenvector of the matrix $\alpha V+W-\frac{\tau}{t}g$, it is also an eigenvector of $\alpha V+W$. Thus
\begin{align}
(\alpha V+W)\xi=\lambda \xi.
\end{align}
Thus
\begin{align}
\bar{\xi}_i R_{i \bar{l}} \left(\alpha V_{l \bar{j}}+W_{l \bar{j}}\right) \xi_j=\bar{\xi}_i R_{i \bar{l}} \lambda \xi_l=\lambda Ric(\xi, \bar{\xi}).
\end{align}
Since $\alpha v_{k\bar{l}}+w_{k\bar{l}}$ is Hermitian, we may choose a local unitary frame $\{e_{1},\cdots, e_{n}\}$ at $(p_{1},t_{1})$ such that each $e_i(p_{1},t_{1})$ is an eigenvector of $\alpha V+W$, with $e_{1}(p_{1},t_{1})=\xi_{p_{1},t_{1}}$ and $\lambda_1 = \lambda$. 
\begin{align}
-\bar{\xi}_{i}R_{i\bar{j}l\bar{k}}(\alpha v_{k\bar{l}}+w_{k\bar{l}})\xi_{j}&=-R(e_{1},\bar{e}_{1},e_{k},\bar{e}_{k})\lambda_{k}.
\end{align}
Hence 
\begin{align}
&-\bar{\xi}_{i}R_{i\bar{j}l\bar{k}}(\alpha v_{k\bar{l}}+w_{k\bar{l}})\xi_{j}
  + \frac12 \bar{\xi}_{i}R_{i\bar{l}}(\alpha v_{l\bar{j}}
  + w_{l\bar{j}})\xi_{j}
  + \frac12\bar{\xi}_{i}R_{l\bar{j}}(\alpha v_{i\bar{l}}+w_{i\bar{l}})\xi_{j}\nonumber\\
&=-R(e_{1},\bar{e}_{1},e_{k},\bar{e}_{k})\lambda_{k}+\lambda_{1}Ric(e_{1},\bar{e}_{1})\nonumber\\
&=\sum_{k=2}^{n}[(\lambda_{1}-\lambda_{k})R(e_{1},\bar{e}_{1},e_{k},\bar{e}_{k})]\nonumber\\
&\geq -K(n\lambda_{1}-\sum_{k=1}^{n}\lambda_{k})\nonumber\\
&\label{main term in P+Q}=-nK\lambda_{1}+K\left(\frac{\alpha \Delta u}{u(1-f)}+\frac{|\nabla u|^{2}}{2u^{2}(1-f)^{2}}\right).
\end{align}
In the above inequality, it suffices to assume that the orthogonal bisectional curvature is bounded from below.
Since the bisectional curvature controls the Ricci curvature, we have $Ric \geq -CK$.
By the Li-Yau differential Harnack inequality \cite{S-Y}, we have 
\begin{align*}
\frac{|\nabla u|^{2}}{u^{2}}-\frac{4u_{t}}{u}\leq \frac{C}{t}+CK,
\end{align*}

Since $u_{t}=\Delta u$, we have 
\begin{align}\label{eq:2.20}
-\frac{\Delta u}{u}\leq \frac{C}{t}+CK.
\end{align}
Since $0<u<A$, we have $0<\frac{1}{1-f}\leq 1$.
Hence 
\begin{align*}
\frac{\Delta u}{u(1-f)}\geq -\frac{C}{t}-CK.
\end{align*}

\begin{align}
R_{i\bar{j}l\bar{k}}w_{k\bar{l}}\bar{\xi}_{i}\xi_{j}&
=R_{i\bar{j}l\bar{k}}\frac{u_{k}u_{\bar{l}}}{u^{2}(1-f)^{2}}\bar{\xi}_{i}\xi_{j}
=\frac{1}{u^{2}(1-f)^{2}}R(\xi,\bar{\xi},\partial u,\bar{\partial} u)\nonumber\\
&\label{gradient u}\geq -K\frac{1}{u^{2}(1-f)^{2}}|\partial u|^{2}
=-\frac{K}{2}\frac{|\nabla u|^{2}}{u^{2}(1-f)^{2}}.
\end{align}

Combining the estimates (\ref{main term in P+Q}) and (\ref{gradient u}), we obtain 
 \begin{align}\label{P+Q term 2A}
\bar{\xi}^{T}(\alpha P+Q)\xi\geq -nK\lambda_{1}+K\frac{\alpha \Delta u}{u(1-f)}\geq -nK\lambda_{1}-K(\frac{C}{t}+CK).
 \end{align}

Thus, at the maximum point $(p_{1},t_{1})$, the inequalities (\ref{eq:2.11}) and (\ref{P+Q term 2A}) imply
\begin{align}
0&\geq L\mu\geq \frac{\lambda_1^{2}}{\alpha^{2}}-\frac{\tau}{t^{2}}
+\bar{\xi}^{T}(\alpha P+Q)\xi\\
&\geq \frac{\lambda_1^{2}}{\alpha^{2}}-\frac{\tau}{t^{2}}-nK\lambda_1-CK^2-\frac{CK}{t}.
\end{align}
Hence, by the Cauchy-Schwarz inequality, at $(p_{1},t_{1})$ we have 
\begin{align*}
\frac{\lambda_1}{\alpha}\leq C\frac{\sqrt{\tau+1}}{t}+CK.
\end{align*}
 Then
\[
\lambda_1-\frac{\tau}{t} \leq(\alpha C\sqrt{\tau+1}-\tau) \frac{1}{t}+\alpha CK \leq \alpha CK \quad \text { at }\left(p_1, t_1\right),
\]
after choosing $\tau$ sufficiently large.

Since $\mu=\lambda_1-\frac{\tau}{t}$ at $\left(p_1, t_1\right)$ is the largest eigenvalue of the $(1,1)$-tensor field $\alpha V+W-\frac{\tau}{t} g$ on $M \times(0, T]$, for any unit tangent vector $\eta \in T_x^{(1,0)} M$, $x \in M$, we have
\[
  \bar{\eta}^T(\alpha V+W) \eta-\frac{\tau}{t} g(\eta, \eta) \leq\left.\left(\lambda_1-\frac{\tau}{t}\right)\right|_{\left(p_1, t_1\right)} \leq \alpha CK \quad \text { in } M \times(0, T) .
\]
Thus
\[
t \bar{\eta}^T V \eta \leq \frac{\tau}{\alpha}+CK t.
\]
This proves part (a) of the theorem.

\textbf{Part (b).} Now we localize the result in part (a). 
By Theorem 4.2 in Chapter 4 of \cite{S-Y}, there exists a cutoff function $\phi\in C_{0}^{\infty}(\mathbb{R})$ supported in $[-2,2]$ such that $1\geq \phi\geq 0$ and $\phi(t)=1$ for $0\leq t\leq 1$. Moreover, $\phi$ satisfies 
\begin{align}
\phi^{'}\leq 0, \phi^{''}\geq -C, \frac{|\phi^{'}|^{2}}{\phi}\leq C
\end{align}
where $C$ is a constant.
Let 
\label{def psi}
$\psi(x,t)=\phi(\frac{2\rho(x)}{R})\phi(\frac{2|t_{0}-t|}{T})$,

where $\rho(x)=d(x_0,x)$ is the distance function. From \cite{S-Y}, we have 
\begin{align}\label{gradient estimate c}
\frac{|\nabla \psi(x,t)|^{2}}{\psi(x,t)}=\frac{|\nabla \phi(\frac{2\rho(x)}{R})|^{2}\phi(\frac{2|t_{0}-t|}{T})}{\phi(\frac{2\rho(x)}{R})}\leq \frac{C}{R^{2}}, 
\end{align}
\begin{align}\label{partial t estimate c}
\frac{|\partial_{t}\psi(x,t)|}{\sqrt{\psi(x,t)}}=\frac{\phi(\frac{2\rho(x)}{R})|\phi^{'}|(\frac{2|t_{0}-t|}{T})\frac{2}{T}}{\sqrt{\psi(x,t)}}\leq \frac{C}{T}.
\end{align}

Since the bisectional curvature is bounded below, we have $Ric\geq -CK$. By the Laplacian comparison theorem, we have 
\begin{align*}
\Delta \rho(x)\leq \frac{C}{\rho(x)}(1+\sqrt{K}\rho)\leq \frac{C}{R}+C\sqrt{K},\end{align*}
for $\rho(x)\geq R$.
Note that
$\phi^{'}(\frac{2\rho(x)}{R})$ vanishes when $\rho(x)\leq R$.
Hence 
\begin{align}\label{laplace comparison c}
\Delta \psi(x,t)\geq -\frac{C}{R^{2}}-\frac{C}{R}\sqrt{K}.
\end{align}

For any smooth function $\eta$, we have
\[
\begin{aligned}
\Delta(\psi \eta)=\Delta \psi \eta+\nabla \psi \cdot \nabla \eta+\psi \Delta \eta.
\end{aligned}
\]
Hence
\[
\begin{aligned}
\psi L \eta= & -\psi \partial_t \eta+\psi \Delta \eta-\psi \frac{f}{1-f} \nabla f \cdot \nabla \eta \\
= & -\partial_t(\psi \eta)+\Delta(\psi \eta)-\frac{f}{1-f} \nabla f \cdot \nabla(\psi \eta)+\eta \partial_t \psi-\eta \Delta \psi \\
& +\frac{f}{1-f} \eta \nabla f \cdot \nabla \psi-\nabla \psi \cdot \nabla \eta.
\end{aligned}
\]
The last term can be written as
\[
\begin{aligned}
\nabla \psi \cdot \nabla \eta & =\frac{\nabla \psi}{\psi} \psi \nabla \eta=\frac{\nabla \psi}{\psi}(\nabla(\psi \eta)-\nabla \psi \eta) \\
& =\frac{\nabla \psi}{\psi} \nabla(\psi \eta)-\frac{|\nabla \psi|^2}{\psi} \eta.
\end{aligned}
\]
Hence
\begin{align}\label{cutoff equation}
\psi L \eta=-\partial_t(\psi \eta)+\Delta(\psi \eta)-\frac{f}{1-f} \nabla f \cdot \nabla(\psi \eta)-\frac{\nabla \psi}{\psi} \nabla(\psi \eta) \\
+\eta \partial_t \psi-\eta \Delta \psi+\eta \frac{ f}{1-f} \nabla f \cdot \nabla \psi+\frac{|\nabla \psi|^2}{\psi} \eta.
\end{align}
Denote
\begin{equation}
  L_1=-\partial_t+\Delta-\frac{f}{1-f} \nabla f \cdot \nabla-\frac{ \nabla \psi}{\psi} \nabla.
\end{equation}
Hence the equation can be written as
\[
\psi L \eta=L_1(\eta \psi)-\eta L_1 \psi,
\]
or
\[
L_1(\eta \psi)=\psi L \eta+\eta L_1 \psi.
\]
With $\lambda$ as introduced in (\ref{def lambda}), we have
\begin{equation}\label{local lambda}
  L_1(\psi \lambda)=\psi L \lambda+\lambda L_1 \psi=\psi\left[H+\bar{\xi}^T(\alpha P+Q) \xi\right]+\lambda L_1 \psi,
\end{equation}
where $H$ is given by (\ref{def H}). We next estimate $L_1 \psi$.
\[
L_1 \psi=-\partial_t \psi+\Delta \psi-\frac{|\nabla \psi|^2}{\psi}-\frac{f}{1-f} \nabla f \cdot \nabla \psi.
\]
For the last term, we write
\[
-\frac{f}{1-f} \nabla f \cdot \nabla \psi=-\frac{f}{1-f} \sqrt{\psi} \nabla f \cdot \frac{\nabla \psi}{\sqrt{\psi}} .
\]
Since $\frac{-f}{1-f}<1$ and $\frac{\nabla \psi}{\sqrt{\psi}}$ is bounded, the estimates (\ref{gradient estimate c}), (\ref{partial t estimate c}), and (\ref{laplace comparison c}) imply
\begin{align}\label{estimate of psi}
L_{1}\psi\geq -\frac{C}{R^{2}}-\frac{C}{R}\sqrt{K}-\frac{C}{T}-\frac{C}{R}\sqrt{\psi} |\nabla f|.
\end{align}

It remains to control
\[
\sqrt{\psi} \nabla f.
\]
Recall that $f$ satisfies
\[
-\partial_t f+\Delta f=-\frac{|\nabla f|^2}{2}.
\]
Then
\[
L f=-\partial_t f+\Delta f-\frac{f}{1-f}|\nabla f|^2=-\frac{|\nabla f|^2}{2}-\frac{f}{1-f}|\nabla f|^2=\frac{-1-f}{2(1-f)}|\nabla f|^2.
\]
Note
\[
\frac{-1-f}{1-f} \geq \frac{1}{2} \text { if } f \leq-3 .
\]
Then
\[
L f \geq \frac{1}{4}|\nabla f|^2 .
\]
Hence, we obtain
\[
L_1 f=L f-\frac{\nabla \psi}{\psi} \nabla f=L f-\frac{\nabla \psi}{\psi\sqrt{\psi}} \sqrt{\psi} \nabla f,
\]
and then
  \begin{align}
    \psi L_1 f & =\psi L f-\frac{\nabla \psi}{\sqrt{\psi}} \sqrt{\psi} \nabla f \geq \frac{1}{4} \psi|\nabla f|^2- \frac{\nabla \psi}{\sqrt{\psi}} \sqrt{\psi} \nabla f \nonumber\\
    & \label{estimate of f}\geq \frac{1}{8} \psi|\nabla f|^2-C \frac{|\nabla \psi|^2}{\psi} .
    \end{align}

Now, for a constant $\beta \in \mathbb{R}^{+}$ to be determined, consider
\[
\psi L_1(\psi \lambda+\beta f)=\psi^2 H+\psi^2 \bar{\xi}^T(\alpha P+Q) \xi+\psi \lambda L_1 \psi+\beta \psi L_1 f .
\]
We consider the eigenvalues of the $(1,1)$-tensor field
\[
\psi(\alpha V+W)+\beta f g .
\]
If $\xi$ is an eigenvector of $\psi(\alpha V+W)+\beta f g$ at some point $(x, t)$ corresponding to the eigenvalue $\mu$, then
\[
[\psi(\alpha V+W)+\beta f g] \xi=\mu \xi,
\]

If $\psi(x, t) \neq 0$, then $\xi$ is also an eigenvector of $\alpha V+W$. Hence
\[
\mu=\psi \lambda+\beta f .
\]
Extend $\xi$ to a vector field around $x$ by parallel translation along geodesics starting from $x$, and still denote this vector field by $\xi$. Define a function $\mu=\mu(x, t)$ around $\left(x, t\right)$ by
\[
\mu=\bar{\xi}^T(\psi(\alpha V+W)+\beta f g) \xi .
\]

Since the domain $\Omega$ is given by
\[
\Omega=Q_{R, T}\left(x_0, t_0\right)=B\left(x_0, R\right) \times\left(t_0-T, t_0\right] .
\]
we have
\[
\left.\mu\right|_{\partial_p \Omega}=\left.\beta f\right|_{\partial_p \Omega}<0,
\]
where $\partial_p \Omega=\partial B\left(x_0, R\right)\times\left(t_0-T, t_0\right]\cup B\left(x_0, R\right)\times \{t_0-T\}$.
We will estimate $\mu$ from above. Recall from (\ref{local lambda}) and $\mu=\psi\lambda+\beta f$ that
\begin{equation}\label{estimate of mu 2}
  \psi L_1 \mu=\psi^2 H+\psi^2 \bar{\xi}^T(\alpha P+Q) \xi+\psi \lambda L_1 \psi+\beta \psi L_1 f .
\end{equation}
From (\ref{estimate H}), we have 
\[
\psi^2 H \geq \frac{(\psi \lambda)^2}{\alpha^2}+\psi^2 \lambda\left(w-\frac{2}{\alpha^2} \bar{\xi}^T W \xi\right)-f \psi^2 \lambda\left(w-\frac{2}{\alpha} \bar{\xi}^T W \xi\right)
\]
Let $\mu_1$ be the largest eigenvalue of $\psi(\alpha V+W)+\beta f g$, with unit eigenvector $\xi$, and suppose that $\mu_1$ is attained at the space time point $\left(p_1, t_1\right)$. Assume that $\mu_1>0$ at $(p_{1},t_{1})$; otherwise, the desired local estimate follows immediately. We then have
\[
\psi \lambda_1+\beta f \geq 0,
\]
and hence $\psi \lambda_1 \geq 0$. Then at the point $\left(p_1, t_1\right)$, we have
\[
\psi^2 H \geq \frac{(\psi \lambda_1)^2}{\alpha^2} .
\]
Combining this with (\ref{estimate of f}), (\ref{estimate of mu 2}), and (\ref{estimate of psi}), we obtain
\[
\begin{gathered}
\psi L_1 \mu \geq \frac{(\psi \lambda_1)^2}{\alpha^2}+\psi^2 \bar{\xi}^T(\alpha P+Q) \xi+\beta\left[\frac{1}{8} \psi|\nabla f|^2-C \frac{|\nabla \psi|^2}{\psi}\right] \\
+\psi \lambda_1(-\frac{C}{R^{2}}-\frac{C}{R}\sqrt{K}-\frac{C}{T}-\frac{C}{R}\sqrt{\psi} |\nabla f|).
\end{gathered}
\]
By the Cauchy-Schwarz inequality, we get
\[
\begin{gathered}
\psi L_1 \mu \geq
\frac{1}{2\alpha^2}(\psi \lambda)^2+\psi^2 \bar{\xi}^T(\alpha P+Q) \xi+\beta\left(\frac{1}{8} \psi|\nabla f|^2-C \frac{|\nabla \psi|^2}{\psi}\right) \\
+\psi \lambda_1(-\frac{C}{R^{2}}-\frac{C}{R}\sqrt{K}-\frac{C}{T})
-\frac{C}{R^{2}}\psi |\nabla f|^{2}
\end{gathered}
\]
We now take
\begin{equation}\label{eq:2.30}
  \beta= \frac{8C}{R^{2}}.
\end{equation}
Hence we have 
\[
\begin{aligned}
\psi L_1 \mu_1 &\geq \frac{1}{2\alpha^2}(\psi \lambda_1)^2+\psi^2 \bar{\xi}^T(\alpha P+Q) \xi \\
&+\psi \lambda_1(-\frac{C}{R^{2}}-\frac{C}{R}\sqrt{K}-\frac{C}{T})
-\frac{C}{R^{4}}
\end{aligned}
\]

At $(p_{1},t_{1})$, $\mu_1$ is the largest eigenvalue of $\psi(\alpha V+W)+\beta f g$ with unit eigenvector $\xi$. Since $f \leq 0$ and $\psi=0$ on the parabolic boundary of $\Omega$, the point $p_1$ lies in the interior of $B(p, R)$.

Since $\left(p_1, t_1\right)$ is a maximum point of $\mu$, we have
\[
\begin{aligned}
0 &\geq \psi L_1 \mu \geq \frac{1}{2\alpha^2}(\psi \lambda_1)^2+\psi^2 \bar{\xi}^T(\alpha P+Q) \xi \\
&+\psi \lambda_1(-\frac{C}{R^{2}}-\frac{C}{R}\sqrt{K}-\frac{C}{T})
-\frac{C}{R^{4}}
\end{aligned}
\]
We now estimate $\psi^2 \bar{\xi}^T(\alpha P+Q) \xi$.

Since $Ric \geq-CK$, \cite{Li-Yau} and our choice of the cutoff function $\psi$ imply that, for any $a>2$,

\begin{equation}\label{Li-Yau local}
\begin{aligned}
-a \psi^2 \frac{\Delta u}{u}\leq \psi^2\left(\frac{|\nabla u|^2}{u^2}-a \frac{u_t}{u}\right) \leq C\left(\frac{1}{T}+\frac{1}{R^2}+K\right) .
\end{aligned}
\end{equation}
We note that the term $1 / T$ appears as $1 / t$ in \cite{Li-Yau}. Since our cutoff function is supported in a smaller cube, these two terms are comparable. This estimate will also be used in the local Hessian estimate on Riemannian manifolds. 
From (\ref{P+Q term 2A}), we obtain 
\[
\psi^2 \bar{\xi}^T (\alpha P +Q)\xi \geq \psi^{2}(-nK\lambda_{1}+K\frac{\alpha \Delta u}{u(1-f)})\geq  -CK \psi^2 \lambda_1 - CK\left(\frac1T + \frac{1}{R^2}\right) -CK^2,
\]

Therefore,
\begin{equation}
\begin{aligned}
   \frac{(\psi \lambda_1)^2}{2\alpha^2} \leq& CK \psi^2 \lambda_1 + CK\left(\frac1T + \frac{1}{R^2}\right) +CK^2 \\
   &+ C\left(\frac{1}{T} + \frac{1}{R^2}\right)^2 +  C \left(\frac{1}{T}+\frac{1}{R^2}\right) \psi \lambda_1,
\end{aligned}
\end{equation}
and hence
\[
\psi \lambda_1 \leq C\left(\frac{1}{T}+\frac{1}{R^2}+K\right).
\]
Since $f<0$, 
\[
\mu_1=\left.\mu\right|_{\left(p_1, t_1\right)}=\left.(\psi \lambda+\beta f)\right|_{\left(p_1, t_1\right)} \leq C\left(\frac{1}{T}+\frac{1}{R^2}+K\right) .
\]
 Therefore
\[
\mu \leq C\left(\frac{1}{T}+\frac{1}{R^2}+K\right) \quad \text { in } Q_{R, T}.
\]
Hence, for any unit tangent vector $\xi$ at $x$ with $(x, t) \in Q_{R, T}$, it holds
\[
\psi \bar{\xi}^T(\alpha V+W) \xi+\beta f \leq C\left(\frac{1}{T}+\frac{1}{R^2}+K\right) \quad \text { in } Q_{R, T},
\]
or
\[
\psi \bar{\xi}^T(\alpha V+W) \xi \leq C\left(\frac{1}{T}+\frac{1}{R^2}+K\right)+\beta|f|, \quad \text { in } Q_{R, T} .
\]
Recalling that $\beta=\frac{8C}{R^2}$, we obtain
\[
\psi \bar{\xi}^T V \xi \leq C\left(\frac{1}{T}+\frac{1}{R^2}+K\right)(1-f),
\]
and hence
\[
\psi \frac{u_{i \bar{j}} \bar{\xi}_i \xi_j}{u} \leq C\left(\frac{1}{T}+\frac{1}{R^2}+K\right)(1-f)^2 .
\]
This gives the desired estimate.
\end{proof}

\section{Riemannian Manifolds}

\begin{proof}
We follow the notation of \cite{Han-Zhang}.
Set $f=\log \frac{u}{A}$, $v_{ij}=\frac{u_{ij}}{u(1-f)}$, $w_{ij}=\frac{u_{i}u_{j}}{u^{2}(1-f)^{2}}$ and 
$$
L=-\partial_{t}+\Delta -\frac{2f}{1-f}\nabla f\cdot \nabla.
$$
Set $V=(v_{ij}), W=(w_{ij}),w=tr(W)$.
From \cite{Han-Zhang}, we have 
\begin{equation*}
LV=(1-f)wV+P,
\end{equation*}
\begin{equation*}
LW=2(1-f)wW+2(V+fW)^{2}+Q,
\end{equation*}
where $P,Q$ are matrices whose $(i,j)$-th components are 
\begin{align}\label{Riemann P}
P_{ij}&=-2R_{kijl}v_{kl}+R_{il}v_{jl}+R_{jl}v_{il}
+(\nabla_{i}R_{jl}+\nabla_{j}R_{il}-\nabla_{l}R_{ij})\frac{u_{l}}{u(1-f)},
\end{align}
and 
\begin{equation}\label{Riemann Q}
Q_{ij}=R_{ik}w_{kj}+R_{jk}w_{ki}.
\end{equation}
Suppose that the 2-tensor 
$$
\alpha V+W-\frac{\tau}{t}g
$$
attains its largest nonnegative eigenvalue at a space time point $(p_{1},t_{1})$, with $t_{1}>0$. Here $\alpha$ and $\tau$ are constants to be determined later. Let $\xi$ be a unit eigenvector at $(p_{1},t_{1})$ corresponding to the largest eigenvalue of the matrix $\alpha V+W-\frac{\tau}{t}g$. We extend $\xi$ to a smooth vector field by parallel translation along geodesics emanating from $p_{1}$. Let
\begin{equation*}
\mu=\xi^{T}(\alpha V+W-\frac{\tau}{t}g)\xi,
\end{equation*}
and 
\begin{equation*}
\lambda=\xi^{T}(\alpha V+W)\xi
\end{equation*}
In a space time neighborhood of $(p_{1},t_{1})$, $\mu$ and $\lambda$ are smooth, and at $(p_{1},t_{1})$ they are the largest eigenvalues of the matrices $\alpha V+W-\frac{\tau}{t_1}g$ and $\alpha V+W$, respectively. From \cite{Han-Zhang}, we have 
\begin{equation}
L\mu =H-\frac{\tau}{t^{2}}+\xi^{T}(\alpha P+Q)\xi,
\end{equation}
where $H=\alpha (1-f)w\xi^{T}V\xi+2(1-f)w\xi^{T}W\xi+2|(V+fW)\xi|^{2}$.
Since $\mu$ attains its nonnegative maximum at $(p_{1},t_{1})$, \cite{Han-Zhang} gives 
\begin{equation}\label{Maximum principle}
0\geq L\mu\geq \frac{2\lambda^{2}}{\alpha^{2}}-\frac{\tau}{t^{2}}+\xi^{T}(\alpha P+Q)\xi , \quad \text{at}\quad (p_1,t_1).
\end{equation}
We note that the estimate of $H$ in \cite{Han-Zhang} does not require any curvature assumptions.
We next estimate $\xi^{T}(\alpha P+Q)\xi$ at $(p_{1},t_{1})$. From (\ref{Riemann P}) and (\ref{Riemann Q}), we have
\begin{align}\label{alpha P+Q Riemann}
(\alpha P+Q)_{ij}&=-2 R_{kijl}(\alpha v_{kl}+w_{kl})+2R_{kijl}w_{kl}+R_{il}(\alpha v_{jl}+w_{jl})+R_{jl}(\alpha v_{il}+w_{il})\nonumber\\
&+(\nabla_{i}R_{jl}+\nabla_{j}R_{il}-\nabla_{l}R_{ij})\frac{u_{l}}{u(1-f)}.
\end{align}
Since, at $(p_{1},t_{1})$, $\xi$ is an eigenvector of the matrix $\alpha V+W-\frac{\tau}{t}g$, it is also an eigenvector of $\alpha V+W$ corresponding to the largest eigenvalue at $(p_1,t_1)$. Thus
\begin{align}
(\alpha V+W)\xi=\lambda_1 \xi
\end{align}
Thus
\begin{align}
\xi_{i}R_{il}(\alpha v_{jl}+w_{jl})\xi_{j}=\xi_{i}R_{il}\lambda_{1} \xi_{l}=\lambda_{1} Ric(\xi,\xi).
\end{align}
Since $\alpha v_{kl}+w_{kl}$ is symmetric, its eigenvalues $\lambda_{i}, 1\leq i\leq n$, at $(p_{1},t_{1})$ are real. Set $\lambda_{1}\geq \lambda_{2}\geq \cdots\geq \lambda_{n}$. The eigenvectors $(\alpha V+W)e_{i}=\lambda_{i}e_{i}$ form an orthonormal frame at $(p_{1},t_{1})$, with $e_{1}=\xi$. We extend $\{e_{1},\cdots,e_{n}\}$ to a smooth orthonormal frame in a neighborhood of $p_{1}$ by parallel translation along geodesics emanating from $p_{1}$. In this orthonormal frame, at $(p_{1},t_{1})$, we have 
\begin{align}
-2\xi_{i}R_{kijl}(\alpha v_{kl}+w_{kl})\xi_{j}&=-2\sum_{k=2}^{n}R(e_{k},e_{1},e_{1},e_{k})\lambda_{k},
\end{align}
and
\begin{align}\label{v term}
&-2\xi_{i}R_{kijl}(\alpha v_{kl}+w_{kl})\xi_{j}+\xi_{i}R_{il}(\alpha v_{jl}+w_{jl})\xi_{j}+\xi_{i}R_{jl}(\alpha v_{il}+w_{il})\xi_{j}\nonumber\\
&=-2\sum_{k=2}^{n}R(e_{k},e_{1},e_{1},e_{k})\lambda_{k}+2\lambda_{1}Ric(e_{1},e_{1})\nonumber\\
&=2\sum_{k=2}^{n}[(\lambda_{1}-\lambda_{k})R(e_{k},e_{1},e_{1},e_{k})]\nonumber\\
&\geq -2K_{1}(n\lambda_{1}-\sum_{k=1}^{n}\lambda_{k})\nonumber\\
&=-2nK_{1}\lambda_{1}+2K_{1}(\frac{\alpha \Delta u}{u(1-f)}+\frac{|\nabla u|^{2}}{u^{2}(1-f)^{2}}).
\end{align}
In the above formula, the inequality follows from $\lambda_{1}-\lambda_{k}\geq 0$ and $sec\geq -K_{1}$.

Since the sectional curvature is bounded from below, $sec\geq -K_{1}$, the Li-Yau differential Harnack inequality (see Chapter 4 of \cite{S-Y}) gives 
\begin{align}\label{local laplace estimate}
\frac{|\nabla u|^{2}}{u^{2}}-2\frac{u_{t}}{u}\leq \frac{c(n)}{t}+c(n)K_{1}.
\end{align}
Since $u_{t}=\Delta u$, we have 
\begin{align*}
-\frac{\Delta u}{u}\leq \frac{c(n)}{t}+c(n)K_{1}.
\end{align*}
Since $0<u<A$,  $0<\frac{1}{1-f}\leq 1$.
Hence 
\begin{align*}
\frac{\Delta u}{u(1-f)}\geq -\frac{c(n)}{t}-c(n)K_{1}.
\end{align*}
For the second term in (\ref{alpha P+Q Riemann}), we obtain
\begin{align}\label{w term}
2R_{kijl}w_{kl}\xi_{i}\xi_{j}&=2R_{kijl}\frac{u_{k}u_{l}}{u^{2}(1-f)^{2}}\xi_{i}\xi_{j}=2\frac{1}{u^{2}(1-f)^{2}}R(\nabla u,\xi,\xi,\nabla u)\nonumber\\
&\geq 2\frac{1}{u^{2}(1-f)^{2}}(-K_{1})|\nabla u-<\nabla u,\xi>\xi|^{2}\geq 2\frac{|\nabla u|^{2}}{u^{2}(1-f)^{2}}(-K_{1}).
\end{align}

Substituting the estimates (\ref{v term}) and (\ref{w term}) into (\ref{alpha P+Q Riemann}), we obtain 
\begin{align}
\label{P+Q term 21}(\alpha P_{ij}+Q_{ij})\xi_{i}\xi_{j}&\geq -2nK_{1}\lambda_{1}+2K_{1}\frac{\alpha \Delta u}{u(1-f)}+\xi_{i}\xi_{j}(\nabla_{i}R_{jl}+\nabla_{j}R_{il}-\nabla_{l}R_{ij})\frac{u_{l}}{u(1-f)}\\
&\label{P+Q term 2}\geq -2nK_{1}\lambda_{1}-2K_{1}(\frac{c(n)}{t}+c(n)K_{1})-c(n)K_{2}\frac{|\nabla u|}{u(1-f)}.
\end{align}

Since the sectional curvature is bounded below, by Hamilton's estimate (Theorem 1.1 in \cite{Hamilton}), we have 
\begin{align}\label{Hamilton term}
\frac{|\nabla u|^{2}}{u^{2}}\leq (\frac{1}{t}+2(n-1)K_{1})\log \frac{A}{u}.
\end{align}

Combining (\ref{P+Q term 2}) and (\ref{Hamilton term}), we obtain 
\begin{align}\label{P+Q term 3}
\xi^{T}(\alpha P+Q)\xi&\geq  -2nK_{1}\lambda_{1}-2K_{1}(\frac{c(n)}{t}+(n-1)K_{1})-c(n)K_{2}\sqrt{\frac{1}{t}+c(n)K_{1}}
\end{align}

Thus, at the maximum point $(p_{1},t_{1})$, inequalities (\ref{Maximum principle}) and (\ref{P+Q term 3}) give 
\begin{align*}
0&\geq L\mu\geq \frac{2\lambda_{1}^{2}}{\alpha^{2}}-\frac{\tau}{t^{2}}+\xi^{T}(\alpha P+Q)\xi\\
&\geq  \frac{2\lambda_{1}^{2}}{\alpha^{2}}-\frac{\tau}{t^{2}}-2nK_{1}\lambda_{1}-2K_{1}(\frac{c(n)}{t}+c(n)K_{1})-c(n)K_{2}\sqrt{\frac{1}{t}+c(n)K_{1}}
\end{align*}
Hence, by the Cauchy-Schwarz inequality, at $(p_{1},t_{1})$ we have 
\begin{align*}
\frac{\lambda}{\alpha}\leq \frac{\sqrt{\tau+1}}{t}+B,
\end{align*}
where $B$ is a nonnegative constant depending on $K_{1}$ and $K_{2}$.
The remaining argument is the same as in \cite{Han-Zhang}; hence we obtain 
$$
t\eta^{T}V\eta\leq \frac{\tau}{\alpha}+Bt,
$$ 
where $\eta$ is a unit vector. 
\end{proof}

For the local version, we again follow the notation of \cite{Han-Zhang}. 
\begin{proof}
Let $\psi$ be a cutoff function from (\ref{def psi})  and 
$$
L_{1}=-\partial_{t}+\Delta -\frac{2f}{1-f}\nabla f\cdot \nabla -\frac{2\nabla \psi}{\psi}\nabla.
$$
Let  
$$
\mu_1=\max_{(x,t)\in Q_{R,T}}\max_{\xi\in T_x M,|\xi|=1}\xi^T (\psi(\alpha V+W)+fg)\xi.
$$ 
Assume that $\mu_1>0$; otherwise, the local estimate follows directly. Since $f<0$ and $\psi=0$ on the parabolic boundary of $Q_{R,T}$, $\mu_1$ is attained at an interior point $(p_1,t_1)$ and on a unit tangent vector $\xi\in T_{p_1}M$. Since $\xi$ is an eigenvector of the matrix $\psi(\alpha V+W)+\beta g$ at $(p_1,t_1)$ corresponding to the largest nonnegative eigenvalue $\mu_1$, $\xi$ is also an eigenvector of the matrix $\alpha V+W$ corresponding to the largest eigenvalue $\lambda_1$. We extend $\xi$ to a smooth vector field, still denoted by $\xi$, by parallel translation along geodesics emanating from $p_1$.

From \cite{Han-Zhang}, we have 
\begin{align}\label{local equation 1}
\psi L_{1}\mu=\psi^{2}H+\psi^{2}\xi^{T}(\alpha P+Q)\xi+\psi\lambda L_{1}\psi+\beta \psi L_{1}f,
\end{align}
where $\beta\in \mathbb{R}^{+}$ is a positive constant to be determined, $\mu=\xi^T(\psi(\alpha V+W)+fg)\xi$, and
$\lambda=\xi^{T}(\alpha V+W)\xi$.

By the estimate of $H$ in \cite{Han-Zhang}, at $(p_1,t_1)$ and for $\alpha>2$, we have
\begin{align}\label{local version 2}
0\geq \psi L_{1}\mu \geq \frac{2}{\alpha^{2}}(\psi \lambda)^{2}+\psi^{2}\xi^{T}(\alpha P+Q)\xi+\psi \lambda L_{1}\psi+\beta \psi L_{1}f.
\end{align}
Again, the estimate of $H$ does not require any curvature assumption.

For the same reason as in the K\"ahler case, we obtain 
\begin{align}\label{estimate of psi Riemann case}
L_{1}\psi\geq -\frac{C}{R^{2}}-\frac{C}{R}\sqrt{K_1}-\frac{C}{T}-\frac{C}{R}\sqrt{\psi} |\nabla f|.
\end{align}
Only the lower bound for the Ricci curvature is needed to obtain this estimate.

Also, from \cite{Han-Zhang}, we have 
\begin{align}\label{estimate of f Riemann case}
\psi L_{1}f\geq \frac{1}{4}\psi |\nabla f|^{2}-C\frac{|\nabla \psi|^{2}}{\psi}.
\end{align}
Combining (\ref{estimate of psi Riemann case}) and (\ref{estimate of f Riemann case}), and choosing $\beta=\frac{C}{R^{2}}$, we have, at $(p_1,t_1)$, 
\begin{align}\label{LU 2}
0\geq \psi L_{1}\mu \geq \frac{1}{\alpha^{2}}(\psi \lambda)^{2}+\psi^{2}\xi^{T}(\alpha P+Q)\xi-C\psi \lambda(\frac{1}{R^2}+\frac{1}{R}\sqrt{K_1}+\frac{1}{T})-\frac{C}{R^4}
\end{align}

By Theorem 1.1 in \cite{Zhang}, we have 
\begin{align}\label{local estimate W}
\psi^2|W|\leq C(\frac{1}{T}+\frac{1}{R^2}+K_1).
\end{align}
This estimate also requires only a lower bound for the Ricci curvature.

Using the local estimate (\ref{Li-Yau local}), (\ref{local estimate W}), and (\ref{P+Q term 21}), we obtain 
\begin{align}\label{local alpha P+Q}
\psi^2\xi^{T}(\alpha P+Q)\xi&\geq \psi^2(-2nK_{1}\lambda_{1}+2K_{1}\frac{\alpha \Delta u}{u(1-f)}+\xi_{i}\xi_{j}(\nabla_{i}R_{jl}+\nabla_{j}R_{il}-\nabla_{l}R_{ij})\frac{u_{l}}{u(1-f)})\nonumber\\
&\geq -2nK_1\psi^2\lambda_1-CK_1(\frac{1}{T}+\frac{1}{R^2}+K_1)-CK_2(\frac{1}{T}+\frac{1}{R^2}+K_1)
\end{align}

Substituting (\ref{local alpha P+Q}) into (\ref{LU 2}), we have, at $(p_1,t_1)$,
\begin{align*}
0&\geq \psi L_{1}\mu\geq \frac{1}{2\alpha^{2}}(\psi \lambda_1)^{2} -2nK_1\psi^2\lambda_1-CK_1(\frac{1}{T}+\frac{1}{R^2}+K_1)-CK_2(\frac{1}{T}+\frac{1}{R^2}+K_1)\\
&-C(\frac{1}{T}+\frac{1}{R^2}+\frac{\sqrt{K_1}}{R})^2
\end{align*}
Hence, at $(p_1,t_1)$, we have
\begin{align*}
\psi \lambda_1\leq C(\frac{1}{T}+\frac{1}{R^2}+K_1+K_2)
\end{align*}
Since $f<0$, for any unit tangent vector $\xi$ at $x$ with $(x,t)\in Q_{R,T}$, we have 
\begin{align*}
\psi\xi^T(\alpha V+W)\xi+\beta f\leq C(\frac{1}{T}+\frac{1}{R^2}+K_1+K_2).
\end{align*}
Since $\beta=\frac{C}{R^2}$, we have, for $(x,t)\in Q_{R,T}$,
\begin{align*}
\psi \xi^{T}V\xi\leq C(\frac{1}{T}+\frac{1}{R^2}+K_1+K_2)(1-f).
\end{align*}
Since $\psi(x,t)=1$ for $(x,t)\in Q_{\frac{R}{2},\frac{T}{2}}$, we have, for any unit vector $\xi\in T_x M, (x,t)\in  Q_{\frac{R}{2},\frac{T}{2}}$,
\begin{align*}
\xi^{T}V\xi\leq C(\frac{1}{T}+\frac{1}{R^2}+K_1+K_2)(1-f).
\end{align*}
\end{proof}

\section{Acknowledgements}
This work is supported by National Natural Science Foundation of China(No.11971358).

\bibliographystyle{plain}
\bibliography{ref}

\end{document}